\documentclass[11pt,twoside]{amsart}
\usepackage{latexsym,amssymb,amsmath}
\usepackage[all]{xy}
\usepackage{tikz,pgfplots}
\usepackage{tikz}
\usetikzlibrary{arrows}
\usetikzlibrary{decorations.markings}
\usepackage{extpfeil}
\usepackage{afterpage}
\usepackage{float}
\usepackage{hyperref}
\usepackage{relsize}

\makeatletter  
\@namedef{subjclassname@2020}{%
\textup{2020} Mathematics Subject Classification}
\makeatother

\numberwithin{equation}{section}
\newtheorem{theorem}{Theorem}[section]
\newtheorem{lemma}[theorem]{Lemma}
\newtheorem{proposition}[theorem]{Proposition}
\newtheorem{corollary}[theorem]{Corollary}

\theoremstyle{definition}
\newtheorem{definition}[theorem]{Definition} 
\newtheorem{procedure}[theorem]{Procedure} 

\newtheorem{remark}[theorem]{Remark}
\newtheorem{example}[theorem]{Example}

\begin{document}

\title[Normally torsion-free edge ideals]
{Normally torsion-free edge ideals of weighted oriented graphs}
\thanks{The first author was supported by a scholarship from CONAHCYT,
Mexico. The second and third authors were supported by SNI, Mexico.}

\author[G. Grisalde]{Gonzalo Grisalde}
\address{
Departamento de
Matem\'aticas\\
Centro de Investigaci\'on y de Estudios
Avanzados del
IPN\\
Apartado Postal
14--740 \\
07000 Mexico City, Mexico
}
\email{gjgrisalde@math.cinvestav.mx}

\author[J. Mart\'inez-Bernal]{Jos\'e Mart\'inez-Bernal}
\address{
Departamento de
Matem\'aticas\\
Centro de Investigaci\'on y de Estudios
Avanzados del
IPN\\
Apartado Postal
14--740 \\
07000 Mexico City, Mexico
}
\email{jmb@math.cinvestav.mx}

\author[R. H. Villarreal]{Rafael H. Villarreal}
\address{
Departamento de
Matem\'aticas\\
Centro de Investigaci\'on y de Estudios
Avanzados del
IPN\\
Apartado Postal
14--740 \\
07000 Mexico City, Mexico
}
\email{vila@math.cinvestav.mx}


\begin{abstract} 
Let $I=I(D)$ be the edge ideal of a weighted oriented graph $D$, let $G$
be the underlying graph of $D$, and let
$I^{(n)}$ be the $n$-th symbolic power of $I$ defined using the
minimal primes of $I$. We prove that $I^2=I^{(2)}$ if and
only if the following conditions hold: (i) every vertex of $D$ with weight greater than $1$ is a sink
and (ii) $G$ has no triangles.  Using a result of Mandal and Pradhan and the
classification of normally torsion-free edge ideals of graphs, we
prove that $I^n=I^{(n)}$ for all $n\geq 1$ if and only if the
following conditions hold: (a) every vertex of $D$
with weight greater than $1$ is a sink and (b) $G$ is bipartite. If
$I$ has no embedded primes, conditions (a) and (b) classify when 
$I$ is normally torsion-free. Using polyhedral geometry and
integral closure, we give necessary conditions for the equality of 
ordinary and symbolic powers of monomial ideals with a minimal
irreducible decomposition. Then, we classify when 
the dual of the edge ideal of a weighted oriented graph is 
normally torsion-free. 
\end{abstract}

\maketitle 

\section{Introduction}\label{intro-section}
Let $G$ be a simple graph with vertex set $V(G)=\{t_1,\ldots,t_s\}$ and edge
set $E(G)$. Let $D$ be a {\it weighted oriented graph\/} whose {\it underlying
graph\/} is $G$, that is, $D$ is a triplet $(V(D),E(D),w)$ where $V(D)=V(G)$,
$E(D)\subset V(D)\times V(D)$ such that 
$$
E(G)=\{\{t_i,t_j\}\mid (t_i,t_j)\in
E(D)\},
$$ 
$|E(D)|=|E(G)|$, and  
$w\colon V(D) \to\mathbb{N}_+$ is a \textit{weight function}. Here $\mathbb{N}_+$ denotes the set 
of positive integers. In other words, $D$ is obtained
from $G$ by assigning a direction to its edges and a weight to its
vertices. A weighted oriented graph is a special type of
digraph \cite{digraphs}. The \textit{vertex set} of $D$ and the \textit{edge set} of $D$
are $V(D)$ and $E(D)$, respectively. The \emph{weight} of $t_i\in V(D)$
is $w(t_i)$ and is denoted simply by $w_i$. The set of vertices
$\{t_i\in V(D)\mid w_i>1\}$ is denoted by
$V^{+}(D)$. We can regard each vertex $t_i$ as a
 variable and consider the polynomial ring
 $S=K[t_1,\ldots,t_s]$ over a ground field $K$. The
 \textit{edge ideal} of $D$, introduced in \cite{cm-oriented-trees,WOG}, 
is the ideal of $S$ given by 
$$I(D):=(\{t_{i}t_{j}^{w_j}\mid (t_{i},t_{j})\in E(D)\}).$$
\quad If $w_i=1$ for each $t_i\in V(D)$, then $I(D)$ is the usual edge
ideal $I(G)$ of the graph $G$ \cite{cm-graphs}, that has been
extensively studied in the literature 
\cite{graphs-rings,Herzog-Hibi-book,edge-ideals,chapter-vantuyl,wolmer-survey,monalg-rev}. The
motivation to study $I(D)$ comes from coding theory, see
\cite[p.~536]{reyes-vila} and \cite[p.~1]{WOG}. If a vertex $t_i$ of
$D$ is a {\it source\/} (i.e., a vertex 
with only outgoing edges) we shall always 
assume that $w_i=1$ because in this case the definition of $I(D)$
 does not depend on the weight of $t_i$ (Remark~\ref{nov30-21}). A
 \emph{sink} vertex of $D$ is a vertex with only incoming edges. This
 notion will play a role in some of our main results. If all
 vertices of $V^+(D)$ are sinks, then $I(D)$ is obtained from $I(G)$ 
by making the change of variables $t_i\rightarrow t_i^{w_i}$ for
$i=1,\ldots,s$, and in this case some of the algebraic properties and
invariants of $I(G)$ are naturally related to those of $I(D)$ (see
\cite[Corollary~5]{cm-oriented-trees},
\cite[Corollary~4.7]{Mandal-Pradhan1}, \cite[Section~3]{lattice-dim1}). 

A prime ideal $\mathfrak{p}$ of $S$ is an \textit{associated prime}
of $I(D)$ if
$(I(D)\colon f)=\mathfrak{p}$, for some $f\in S$, where 
$(I(D)\colon f):=\{g\in S\mid gf\in I(D)\}$ is
an ideal quotient \cite[p.~8]{AM}. An associated prime of
$I(D)$ which properly contains another associated prime of $I(D)$ is
called an \textit{embedded prime} of $I(D)$. Let
$\mathfrak{p}_1,\ldots,{\mathfrak p}_r$ be   
the minimal primes of $I(D)$, that is, the non-embedded
associated primes of $I(D)$. Given an integer $n\geq 1$,  
the $n$-th {\it symbolic power} of 
$I(D)$, denoted $I(D)^{(n)}$, is the ideal  
$$
I(D)^{(n)}:=\bigcap_{i=1}^r Q_i=\bigcap_{i=1}^r
(I(D)^nS_{\mathfrak{p}_i}\textstyle\bigcap
S),
$$
where $Q_i=I(D)^nS_{\mathfrak{p}_i}\textstyle\cap
S$ is the ${\mathfrak p}_i$-primary component of
$I(D)^n$ and $S_{\mathfrak{p}_i}$ is the localization of $S$ at
$\mathfrak{p}_i$ (see \cite[p.~484]{simis-trung},
\cite[Definition~3.5.1]{Vas1}). In particular, $I(D)^{(1)}$ is the intersection of the non-embedded 
primary components of $I(D)$. An alternative
notion of symbolic 
power can 
be introduced 
using the set ${\rm Ass}(I(D))$ of associated
primes of $I(D)$ instead 
(see, e.g., \cite{cooper-etal,symbolic-powers-survey}): 
$$
I(D)^{\langle n\rangle}: =\bigcap_{\mathfrak{p}\in {\rm
Ass}(I(D))}(I(D)^nS_\mathfrak{p}{\textstyle\bigcap
S)}=\bigcap_{\mathfrak{p}\in {\rm maxAss}(I(D))}
(I(D)^nS_\mathfrak{p}\textstyle\bigcap
S),
$$
where ${\rm maxAss}(I(D))$ denotes the set of maximal elements 
of ${\rm Ass}(I(D))$ (maximal with respect to inclusion). 
Clearly $I(D)^n\subset I(D)^{\langle n\rangle}\subset I(D)^{(n)}$. 
If $I(D)$ has no embedded primes, the two definitions of symbolic powers
coincide. If all vertices of $V^+(D)$ are sinks, then $I(D)$ has no
embedded primes \cite[Lemma~47]{WOG}, $I(D)^{\langle
n\rangle}=I(D)^{(n)}$ for all $n\geq 1$, and $I(D)^n=I(D)^{(n)}$ 
if and only if $I(G)^n=I(G)^{(n)}$ for each $n\geq 1$
\cite[Corollary~4.7]{Mandal-Pradhan1}.   

One of the early works on symbolic powers of monomial ideals was
written by Simis \cite{aron-hoyos}. Giving a combinatorial
characterization of the equality of all ordinary and 
symbolic powers of a monomial ideal is a wide open problem in this
area. This problem has been solved for squarefree monomial ideals 
and for edge ideals of graphs using combinatorial optimization and
graph theory, see \cite[Corollary~3.14]{clutters}, \cite[Theorem
1.4]{hhtz}, and \cite[Theorem~5.9]{ITG}. 

We determine when $I(D)^2$ is equal to
$I(D)^{(2)}$ in terms of the cycles of $G$ and the sinks of $D$, and
give a combinatorial classification for the
equality ``$I(D)^n=I(D)^{(n)}$ for $n \geq 1$''. It is an open problem
to classify the equality ``$I(D)^n=I(D)^{\langle n\rangle}$ for $n
\geq 1$''; for some of the advances to solve this problem see
\cite{Banerjee-etal,Banerjee-Das-Selvaraja,Mandal-Pradhan,Mandal-Pradhan1} and the references therein.   

We come to one of our main results.

\noindent \textbf{Theorem~\ref{I2=I(2)}.}\textit{ 
Let $D$ be a weighted oriented graph and let $G$ be its underlying 
graph. Then, $I(D)^2=I(D)^{(2)}$ if and only if the following
two conditions hold:
\begin{enumerate}
\item[(i)] Every vertex of $V^+(D)$ is a sink;
\item[(ii)] $G$ has no triangles.
\end{enumerate}
}

If every vertex in $V^+(D)$ is a sink and $G$ is bipartite,
then $I(D)^{n}=I(D)^{\langle n\rangle}=I(D)^{(n)}$ for all $n\geq 1$ 
\cite[Corollary 3.8]{Mandal-Pradhan}. One of our main results shows that 
the converse holds.  

\noindent \textbf{Theorem~\ref{In=I(n)}.}\textit{  
Let $D$ be a weighted oriented graph and let $G$ be its underlying 
graph. Then, $I(D)^n=I(D)^{(n)}$ for all $n\geq 1$ if and only if the following
two conditions hold:
\begin{enumerate}
\item[{\rm (a)}] Every vertex in $V^+(D)$ is a sink;
\item[{\rm (b)}] $G$ is a bipartite graph.
\end{enumerate}
}

As a consequence, if $I(D)$ has no embedded primes,  then 
${\rm Ass}(I(D)^n)={\rm Ass}(I(D))$ for all $n\geq 1$ (i.e., $I(D)$ is
normally torsion-free) if and only if every vertex in $V^{+}(D)$ is
a sink and $G$ is a bipartite graph (Corollary~\ref{ntf-weighted}).

In Section~\ref{symbolic-section2}, using polyhedral geometry and
integral closure,  
we give necessary conditions for the equality of 
ordinary and symbolic powers of monomial ideals with a minimal
irreducible decomposition. 
To explain our result, we
introduce some more notation. 

An ideal $L$ of $S$ is called {\it irreducible} if 
$L$ cannot be written as an intersection of two ideals of $S$ that
properly contain $L$. Given $b=(b_1,\ldots,b_s)$ in
$\mathbb{N}^s\setminus\{0\}$, where $\mathbb{N}=\{0,1,\ldots\}$, we set
$\mathfrak{q}_b:=(\{t_i^{b_i}\vert\, b_i\geq 1\})$ and
$b^{-1}:=\sum_{b_i\geq 1}b_i^{-1}e_i$, where $e_i$ denotes the $i$-th unit
vector in $\mathbb{R}^s$. Let $I$ be a monomial ideal of
$S$. According to 
\cite[Theorems~6.1.16 and 6.1.17]{monalg-rev}, there is a 
\textit{unique irreducible decomposition}:
\begin{equation}\label{jun4-21}
I=\mathfrak{q}_{1}\textstyle\cap\cdots\cap\mathfrak{q}_{m},
\end{equation}
where each $\mathfrak{q}_{i}$ is an irreducible monomial ideal 
of the form $\mathfrak{q}_i=\mathfrak{q}_{\alpha_i}$ for some
$\alpha_i\in\mathbb{N}^s\setminus\{0\}$, and $I\neq\textstyle\bigcap_{i\neq j}\mathfrak{q}_{i}$ for
$j=1,\ldots,m$. The ideals $\mathfrak{q}_{1},\ldots,\mathfrak{q}_{m}$
are the {\it irreducible components\/} of $I$. The vectors
$\alpha_1^{-1},\ldots,\alpha_m^{-1}$ are used below to define the 
irreducible polyhedron of $I$.

Since irreducible ideals are
primary, the irreducible decomposition of $I$ is a 
primary decomposition of $I$.  The irreducible decomposition of $I$
is \textit{minimal} if ${\rm rad}(\mathfrak{q}_i)\neq{\rm
rad}(\mathfrak{q}_j)$ for $i\neq j$. For edge ideals of
weighted oriented graphs and for squarefree monomial ideals, their irreducible
decompositions are minimal \cite{WOG,monalg-rev} (cf. Theorem~\ref{pitones-reyes-toledo}).

The monomials of $S$ are denoted
by $t^a:=t_1^{a_1}\cdots t_s^{a_s}$, $a=(a_1,\dots,a_s)$ in
$\mathbb{N}^s$. We denote the minimal set of generators of $I$ by
$\mathcal{G}(I):=\{t^{v_1},\ldots,t^{v_q}\}$. The \textit{incidence matrix} of
the ideal $I$ is the $s\times q$ matrix $A$ with column vectors $v_1,\ldots,v_q$. The
$\textit{covering polyhedron}$ of $I$, denoted by
$\mathcal{Q}(I)$, is the rational polyhedron
$$
\mathcal{Q}(I):=\{x\vert\, x\geq 0;\,xA\geq 1\},
$$
where $1=(1,\ldots,1)$. The \textit{Newton
polyhedron} of $I$, denoted ${\rm NP}(I)$, is the 
integral polyhedron 
\begin{equation}\label{NP-def}
{\rm NP}(I):=\mathbb{R}_+^s+{\rm 
conv}(v_1,\ldots,v_q),
\end{equation}
where $\mathbb{R}_+=\{\lambda\in\mathbb{R}\vert\, \lambda\geq
0\}$. 
This polyhedron is the convex hull of the set of all $a\in\mathbb{N}^s$ such
that $t^a\in I$ \cite[p.~141]{Eisen}. The \textit{integral closure} of
 $I^n$ can be described as 
\begin{equation}\label{jun21-21}
\overline{I^n}=
(\{t^a\vert\, a/n\in{\rm NP}(I)\})
\end{equation}
for all $n\geq 1$ \cite[Proposition~3.5(a)]{reesclu}.  
If $I^n=\overline{I^n}$ for all $n\geq 1$, $I$ is said to be
\textit{normal}. Let $\alpha_1,\ldots,\alpha_m$ be the vectors in
$\mathbb{N}^s\setminus\{0\}$ 
associated to the irreducible decomposition of $I$ and let $B$ be the matrix with column vectors
$\alpha_1^{-1},\ldots,\alpha_m^{-1}$. The polyhedron
$$
\{x\vert\, x\geq 0;\,xB\geq 1\},
$$
is called the \textit{irreducible polyhedron} of $I$
 and is denoted by $\mathcal{Q}(B)$ or ${\rm IP}(I)$ \cite{Seceleanu-convex-bodies}.

We come to another of our results.

\noindent \textbf{Theorem~\ref{dec12-21}.}\textit{ Let $I$ be a monomial ideal of $S$
with a minimal irreducible decomposition 
$I=\mathfrak{q}_1\cap\cdots\cap\mathfrak{q}_m$. If $I^{n}=I^{(n)}$ for all
$n\geq 1$, then the following hold:
\begin{enumerate}
\item[(a)] $\overline{I^n}=
\overline{\mathfrak{q}_1^n}\cap\cdots\cap\overline{\mathfrak{q}_m^n}$
for all $n\geq 1$;
\item[(b)] ${\rm NP}(I)={\rm IP}(I)$;
\item[(c)] The vertices of $\mathcal{Q}(I)$ are precisely
$\alpha_1^{-1},\ldots,\alpha_m^{-1}$.
\end{enumerate}
}

As an application, we classify when 
the dual of the edge ideal $I(D)$ of a weighted oriented graph $D$ is 
normally torsion-free. Following
\cite[p.~495]{cm-oriented-trees}, define the 
\textit{dual} of $I(D)$, 
denoted $J(D)$, as the
intersection of all ideals $(t_i,t_j^{w_j})$ such
that $(t_i,t_j)\in E(D)$. Thus
$$
J(D)=\bigcap_{(t_i,t_j)\in E(D)}
(t_i,t_j^{w_j}),
$$
and this is the irreducible decomposition of $J(D)$. There are other
related ways, introduced by Ezra Miller
\cite{ezra-miller,cca}, to define the dual of a monomial ideal. If $w_i=1$ for
all $i$, then $J(D)$ is normally torsion-free if and
only if $G$ is bipartite \cite[Corollary~3.17, Theorem~4.6,
Proposition~4.27]{reesclu}. 

\noindent \textbf{Corollary~\ref{ntf-J}.}\textit{
Let $J(D)$ be the dual of
$I(D)$. Then, $J(D)^{(n)}=J(D)^n$ for all $n\geq 1$ if 
and only if $J(D)$ is normal and ${\rm NP}(J(D))={\rm IP}(J(D))$.
}

In Section~\ref{examples-section}, we present examples related to 
some of our results. Then, in Appendix~\ref{Appendix}, we give the
procedures for \textit{Normaliz}
\cite{normaliz2} and \textit{Macaulay}$2$ \cite{mac2} that are used in the examples 
to compute the symbolic powers of a monomial ideal and its irreducible
decomposition, the vertices of covering polyhedra, and the linear
constraints that define Newton polyhedra. 

For all unexplained
terminology and additional information,  we refer to \cite{AM} for primary
decompositions,  
\cite{digraphs} for the theory of digraphs,  \cite{diestel,Har} for
the theory of graphs,  
and \cite{graphs-rings,Herzog-Hibi-book,edge-ideals,monalg-rev} for the theory of
edge ideals of graphs and monomial ideals.

\section{Preliminaries}\label{prelim-section}
In this section we give some definitions and present some well-known
results that will be used in the following sections.  
To avoid repetitions, we continue to employ 
the notations and
definitions used in Section~\ref{intro-section}.

Let $D = (V(D),E(D),w)$ be a weighted oriented graph
with vertex set $V(D)=\{t_1,\ldots,t_s\}$, underlying graph $G$, and
edge ideal $I(D)$.

\begin{remark}\label{nov30-21} Consider the weighted
oriented graph $D' = (V(D),E(D),w')$ with $w'(t_i) = 1$ if $t_i$ is a
\textit{source} vertex and $w'(t_i) =
w(t_i)$ if $t_i$ is not a source vertex. 
Then, $I(D') = I(D)$, that is, $I(D)$ does not depend on the weights
that we place at source vertices. For this reason we will always 
assume that all sources of $D$ have weight $1$.
\end{remark}

\begin{lemma}\label{WOG-lemma-mandal}  If all vertices of $V^+(D)$ are
sinks, then the following hold:
\begin{enumerate}
\item[(a)] \cite[Lemma~47]{WOG} $I(D)$ has no embedded primes; 
\item[(b)] \cite[Corollary~4.7]{Mandal-Pradhan1} $I(G)^{(s)} = I(G)^s$ if
and only if 
$I(D)^{(s)} = I(D)^s$ for each $s\geq 1$.
\end{enumerate}
\end{lemma}

Let $C$ be a vertex cover of $G$, i.e., a set of vertices of $G$ that
contain at least one vertex of each edge of $G$. A 
\textit{minimal vertex cover} of $G$ is a vertex cover which
is minimal with respect to inclusion.
Following \cite{WOG}, we 
consider the sets 
\begin{enumerate}
\item[\ ] $L_1(C)=\{x\in C\mid\mbox{there is }\, (x,y)\in E(D)\mbox{ with
}y\notin C\}$,
\item[\ ] $L_3(C)=\{x\in C\mid N_D(x)\subset C\}$, where $N_D(x):=N_G(x)$ is
the neighbor set of $x$, and 
\item[\ ] $L_2(C)=C\setminus(
L_1(C)\cup L_3(C))$.
\end{enumerate}
\quad Note that $\{L_i(C)\}_{i=1}^3$ is a partition of $C$. 
A vertex cover $C$ of $G$ is called a {\it strong
vertex cover} of $D$ if $C$ is a minimal vertex cover of
$G$ or else for
all $x\in L_3(C)$ there is $(y,x)\in E(D)$ such that $y\in
L_2(C)\cup L_3(C)$ with $w(y)\geq 2$. 

\begin{figure}[H]
\centering
 \begin{tikzpicture}[scale=0.6]
        \draw rectangle (15,10);
        \draw (7.5,5) ellipse [x radius=5cm, y radius=3cm];
                   \begin{scope}[decoration={markings, mark=at position 0.5 with {\arrow{>}}}] 
                 \draw[postaction={decorate},line width=0.2mm] (1.5,3) -- (4.5,4);
  \end{scope}
  \filldraw[blue] (4.5,4) circle (3pt);
 \draw[-> ,line width=0.2mm] (5.5,3) -- (3,1);
   \begin{scope}[decoration={markings, mark=at position 0.5 with {\arrow{>}}}] 
 \draw[postaction={decorate},line width=0.2mm] (5.6,1) -- (5.5,3);
  \end{scope}
 \filldraw[red] (5.5,3) circle (3pt);                           
  \draw[-> ,line width=0.2mm] (8.5,6) -- (10,6.5);
   \draw[-> ,line width=0.2mm] (8.5,6) -- (10,8.5);
    \begin{scope}[decoration={markings, mark=at position 0.5 with {\arrow{>}}}] 
  \draw[postaction={decorate},line width=0.2mm] (9,9)--(8.5,6);
    \draw[postaction={decorate},line width=0.2mm] (7.5,6.5) -- (8.5,6);
  \end{scope}
   \filldraw[red] (8.5,6) circle (3pt);       
          \draw[violet,dashed] (6,6) ellipse [x radius=1cm, y radius=1.5cm, rotate=10];   
          \draw[-> ,line width=0.2mm] (4,6) -- (6,7.3);
            \draw[-> ,line width=0.2mm] (4,6) -- (6.5,6);
              \begin{scope}[decoration={markings, mark=at position 0.5 with {\arrow{>}}}] 
                 \draw[postaction={decorate},line width=0.2mm] (5.5,6.5)--(4,6);
  \end{scope}
          \node at (6.1,5.2) {\tiny{$N_{D}(x)$}};
                \node at (3.7,6.2) {\small{$x$}};
                 \filldraw[green] (4,6) circle (3pt); 
\begin{scope}[decoration={markings, mark=at position 0.5 with {\arrow{>}}}] 
   \draw[postaction={decorate},line width=0.2mm] (10,1.5) -- (8,3);
     \draw[postaction={decorate},line width=0.2mm] (8.3,1.5) -- (8,3);
  \end{scope}
    \filldraw[blue] (8,3) circle (3pt);
      \draw[-> ,line width=0.2mm] (9.5,4) -- (7.5,4);
          \draw[-> ,line width=0.2mm] (9.5,4) -- (10,3);
\begin{scope}[decoration={markings, mark=at position 0.5 with {\arrow{>}}}] 
\draw[postaction={decorate},line width=0.2mm] (8.5,4.7) -- (9.5,4);
  \end{scope}
   \filldraw[green] (9.5,4) circle (3pt);
\draw[-> ,line width=0.2mm] (11,5) -- (11.5,4);
\begin{scope}[decoration={markings, mark=at position 0.5 with {\arrow{>}}}] 
\draw[postaction={decorate},line width=0.2mm] (14,4)--(11,5);
   \draw[postaction={decorate},line width=0.2mm] (11.5,6)--(11,5);
       \draw[postaction={decorate},line width=0.2mm] (13.3,5.5)--(11,5);
  \end{scope}
   \filldraw[blue] (11,5) circle (3pt);
                    \node at (3,5) {$C$};
                      \node at (1,.7) {$V(D)$};
                 \node at (14,9) {\textcolor{red}{$L_{1}(C)$}};
                  \node at (14,8) {\textcolor{blue}{$L_{2}(C)$}};
                   \node at (14,7) {\textcolor{green}{$L_{3}(C)$}};
\end{tikzpicture}
\caption{The partition $\{L_i(C)\}_{i=1}^3$ of
$C$.\quad\quad\quad\quad\quad\quad\quad\quad\quad}\label{partition-C} 
\end{figure}
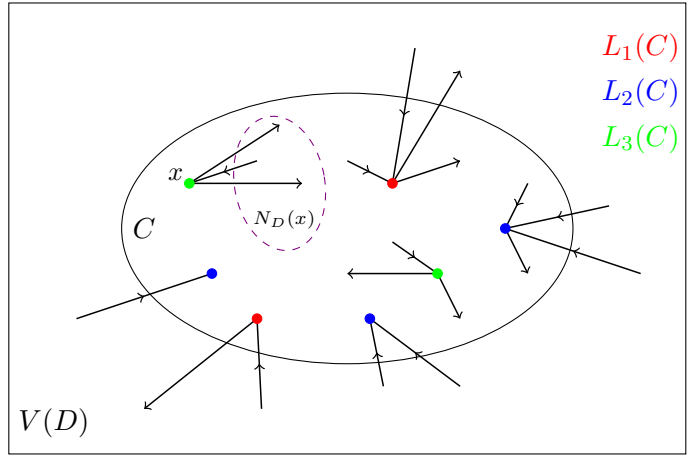
\vspace{-5mm}

\begin{theorem}{\rm\cite[Theorem~25]{WOG}}\label{pitones-reyes-toledo} If
$D$ is a weighted oriented  
graph and $\Upsilon(D)$ is the set of all strong vertex covers of 
$D$, then the irreducible decomposition of $I(D)$
is 
$$
I(D)=\bigcap_{C\in\Upsilon(D)}I_C, 
$$
where $I_C=(L_1(C)\cup\{t_i^{w_i}\vert\, t_i\in L_2(C)\cup L_3(C)\})$. 
\end{theorem}

\begin{corollary}{\rm\cite{WOG}}\label{prt-main} Let $D$ be a weighted 
oriented graph. Then, $\mathfrak{p}$ is an associated prime of
$I(D)$ if and only if $\mathfrak{p}=(C)$ for some strong
vertex cover $C$ of $D$.
\end{corollary}

\begin{proposition}{\rm(\cite[Lemma~3.1]{Banerjee-etal}, \cite[Theorem~25]{WOG}, \cite[Lemma
2.18]{Mandal-Pradhan1})}\label{dec5-21} Let $D$ be a weighted
oriented graph with vertex set $V(D)=\{t_1,\ldots,t_s\}$. The
following conditions are equivalent.
\begin{enumerate}
\item[(a)] $\mathfrak{m}=(t_1,\ldots,t_s)$ is an associated prime of
$I(D)$;
\item[(b)] $V(D)$ is a strong cover of $D$;
\item[(c)] $N_D^+(V^+(D)):=\{x\in V(D)\mid\mbox{there is }y\in V^{+}(D)\mbox{ such that
}(y,x)\in E(D)\}=V(D)$.
\end{enumerate}
\end{proposition}

The set $N_D^+(V^+(D))$ consists of all the outgoing neighbors of
$V^+(D)$. It would be interesting to find 
a structure theorem for oriented graphs that satisfy the equality
$N_D^+(V^+(D))=V(D)$.

\begin{lemma}\cite[p.~169]{Vas1}\label{icd}
If $I$ is a monomial ideal of $S$ and $n\in\mathbb{N}_+$, then 
\begin{equation*}
\overline{I^n}=(\{t^a\in S\mid (t^a)^{p}\in I^{pn}
\mbox{ for some }p\geq 1\}).
\end{equation*}
\end{lemma}

\begin{proof} This follows from the description of the integral
closure given in Eq.~\eqref{jun21-21}.
\end{proof}

\begin{lemma}{\rm(\cite[Lemma~2]{cm-oriented-trees},
\cite[Lemma~3.1]{herzog-hibi-trung})} 
\label{anoth-one-char-spow-general} 
Let $I$ be a monomial ideal of $S$. If 
$\mathfrak{I}_1,\ldots,\mathfrak{I}_r$ are the primary components
corresponding to the  minimal primes 
of $I$, then 
$$
I^{(n)}={\mathfrak I}_1^{n}\textstyle\cap\cdots\cap {\mathfrak I}_r^{n}\ 
\mbox{ for all }\ n\geq 1.
$$
\end{lemma}

\begin{lemma}\label{nov1-21}
If $\mathfrak{p}_1,\ldots,\mathfrak{p}_r$ are the minimal
primes of $I(D)$, then for each $\mathfrak{p}_i$ there is only one 
irreducible component $\mathfrak{q}_i$ of $I(D)$ with ${\rm
rad}(\mathfrak{q}_i)=\mathfrak{p}_i$, and $I(D)^{(n)}=\bigcap_{i=1}^r\mathfrak{q}_i^n$ for all $n\geq 1$. 
\end{lemma}

\begin{proof} By Theorem~\ref{pitones-reyes-toledo}, 
for each $1\leq i\leq r$ there is a unique
irreducible component $\mathfrak{q}_i$ of $I(D)$ whose radical is
$\mathfrak{p}_i$. Hence, by Lemma~\ref{anoth-one-char-spow-general}, one has
$I(D)^{(n)}=\bigcap_{i=1}^r\mathfrak{q}_i^n$ for all $n\geq 1$. 
\end{proof}

Recall that the unique minimal set of generators of a monomial ideal $I$, 
consisting of monomials, is denoted by $\mathcal{G}(I)$. 

\begin{lemma}\cite[Lemma~1]{cm-oriented-trees}\label{duality-of-exponents} 
Let $I\subset S$ be a monomial ideal, with 
$\mathcal{G}(I)=\{t^{v_1},\ldots,t^{v_q}\}$ and
$v_i=(v_{i,1},\ldots,v_{i,s})$ for $i=1,\ldots,q$, and let
$I=\bigcap_{i=1}^m\mathfrak{q}_i$ be its irreducible decomposition. Then
$$
\{t_j^{v_{i,j}}\vert\, v_{i,j}\geq
1\}=\mathcal{G}(\mathfrak{q}_1)\textstyle\cup\cdots\cup\mathcal{G}(\mathfrak{q}_m).
$$ 
\end{lemma}

\begin{definition}\rm An ideal $I$ of $S$ is 
{\em normally torsion-free}
 if ${\rm Ass}(I^n)={\rm Ass}(I)$ for all 
$n\geq 1$. 
\end{definition}

\begin{proposition}\cite[Proposition
4.3.29]{monalg-rev}\label{ntf-char} 
Let $I$ be an ideal of $S$. If $I$ has no embedded 
 primes, then $I$ is normally torsion-free if and only 
if $I^n=I^{(n)}$ for all $n\geq 1$.
\end{proposition}

\begin{corollary}\label{ntf-char-coro} 
Let $I$ be an ideal of $S$. Then, $I^n=I^{(n)}$ for all $n\geq 1$ if and only 
if $I$ has no embedded primes and ${\rm Ass}(I^n)={\rm Ass}(I)$ for all $n\geq 1$.
\end{corollary}

\begin{lemma}\label{powers-m} Let $\mathfrak{m}=(t_1,\ldots,t_s)$ be
the irrelevant maximal ideal of $S$ and let $I\subset S$ be a graded
ideal. Then, the following hold:
\begin{enumerate}
\item[(a)] $IS_\mathfrak{m}\cap S=I$;
\item[(b)] If $\mathfrak{m}\in{\rm Ass}(I)$, then 
$I^{\langle n\rangle}=I^nS_\mathfrak{m}\cap S=I^n$ for all $n\geq 1$.
\end{enumerate}
\end{lemma} 

\begin{proof} (a) Clearly $IS_\mathfrak{m}\cap S\supset I$. To show
the reverse inclusion take $f\in IS_\mathfrak{m}\cap S$. Then, $f=g/h$,
$g\in I$, $h\notin\mathfrak{m}$. Thus, $hf\in I$. Pick a primary
decomposition $I=\bigcap_{i=1}^\ell Q_i$, where the $Q_i$'s are
graded. Then, $hf\in\ Q_i$ for
all $i$. If $f\notin Q_i$ for some $i\in\{1,\ldots,\ell\}$, 
then $h^p\in Q_i$ for some $p\geq 1$. Since
$Q_i$ is graded, $h^p\in Q_i\subset\mathfrak{m}$. Thus,
$h\in\mathfrak{m}$, a contradiction. This proves that $f\in Q_i$ for
all $i$, that is, $f\in I$.

(b) Note that ${\rm maxAss}(I)=\{\mathfrak{m}\}$. Then, by part (a), we
get $I^{\langle n\rangle}=I^nS_\mathfrak{m}\cap 
S =I^n$.
\end{proof}

\begin{proposition}\cite[Proposition~3.6]{cooper-etal}\label{nov6-21}
Let $I\subset S$ be a monomial ideal. If $\mathfrak{p}\in {\rm
Ass}(I)$, then $
I^nS_\mathfrak{p}\cap S=(IS_\mathfrak{p}\cap S)^n$ for all 
$n\geq 1$.
\end{proposition}

\begin{proposition}\label{np-qa} Let
$I$ be a monomial ideal of $S$, let $u_1,\ldots,u_r$ be 
the vertices of $\mathcal{Q}(I)$, and let $\overline{B}$ be the matrix with
column vectors $u_1,\ldots,u_r$. The following hold.
\begin{enumerate}
\item[(a)] \cite[Proposition~3.5(b)]{reesclu} ${\rm NP}(I)=\mathcal{Q}(\overline{B})=\{x\vert\, x\geq 0;\,
x\overline{B}\geq 1\}$; 
\item[(b)] If $I=(t^{v_1},\ldots,t^{v_q})$,
then the vertices of ${\rm NP}(I)$ are
contained in $\{v_1,\ldots,v_q\}$. 
\end{enumerate}
\end{proposition}
\begin{proof} (b) Since ${\rm NP}(I)=\mathbb{R}_+^s+{\rm
conv}(v_1,\ldots,v_q)$, by \cite[Propositions 1.1.36 and
1.1.39]{monalg-rev}, the vertices of ${\rm NP}(I)$ are contained 
in the set $\{v_1,\ldots,v_q\}$. 
\end{proof}

The following result shows that the Cohen--Macaulay property of the
edge ideal of a weighted oriented graph is independent of the weights 
we assign to sinks.   

\begin{proposition}\cite[Lemma~4]{cm-oriented-trees}\label{change-of-variable}  
Let $I\subset S$ be a monomial ideal, with 
$\mathcal{G}(I)=\{t^{v_1},\ldots,t^{v_q}\}$ and
$v_i=(v_{i,1},\ldots,v_{i,s})$ for $i=1,\ldots,q$. Suppose there are
$k$ and $n$ such that $v_{i,k}=1$ for $i=1,\ldots,n$ and $v_{i,k}=0$ for
$i>n$. Let $w_k\in\mathbb{N}_+$ be a weight for $t_k$. If $u$ is a new
variable and $J$ is the ideal of $S[u]$ generated by the monomials 
obtained from $\mathcal{G}(I)$ by replacing $t_k$ by $u^{w_k}$, then $I$ is
Cohen--Macaulay if and only if $J$ is Cohen--Macaulay.
\end{proposition}

\begin{proof} To simplify notation we assume that $k=s$. We grade
$S[u]$ by $\deg(t_s)=w_s$, $\deg(t_i)=1$ for $i\neq s$, and
$\deg(u)=1$. Let $\prec$ be the graded reverse lexicographical order on $S[u]$. In this order
$t_1\succ\cdots\succ t_s\succ u$. We set $f=t_s-u^{w_s}$. This
polynomial is homogeneous of degree $w_s$ and its leading 
term is $t_s$. From the equalities
$$
t^{v_i}-t_1^{v_{i,1}}\cdots
t_{s-1}^{v_{i,s-1}}(t_s-u^{w_s})=t_1^{v_{i,1}}\cdots
t_{s-1}^{v_{i,s-1}} u^{w_s},\ i=1,\ldots,n, 
$$
we obtain $(I,f)=(J,f)$, and $\mathcal{G}(J)\cup\{f\}$ is a Gr\"obner basis of
$(I,f)$. The polynomial $f$ and the variable $t_s$ are both regular on
$S[u]/J$ because the variable $t_s$ does not appear in any minimal
generator of $J$, and $f$ and $u$ are both regular on $S[u]/I$ because
$u$ does not appear in $t^{v_i}$ for $i=1,\ldots,q$. Then, by
\cite[Proposition~2.3.12]{monalg-rev}, the
following conditions are equivalent:
\begin{enumerate}
\item $K[t_1,\ldots,t_{s-1},u]/J$ is Cohen--Macaulay; 
\item $S[u]/J$ is Cohen--Macaulay;
\item $S[u]/(J,t_s-u^{w_s})=S[u]/(I,t_s-u^{w_s})$ is Cohen--Macaulay;
\item $S[u]/I$ is Cohen--Macaulay;
\item $S/I$ is Cohen--Macaulay;
\end{enumerate}
and the proof is complete.
\end{proof}

\section{Equality of ordinary and symbolic powers of edge ideals}\label{symbolic-section}

Let $S=K[t_1,\ldots,t_s]$ be a polynomial ring over
a field $K$, let $D$ be a weighted oriented graph with vertex set
$V(D)=\{t_1,\ldots,t_s\}$ whose underlying
graph is $G$, and let $I(D)$ be the edge ideal of $D$. In this
section, we give combinatorial classifications of the equality
$I(D)^2=I(D)^{(2)}$, the equality of all ordinary and symbolic
powers of $I(D)$, and the torsion-freeness of $I(D)$.  
To avoid repetitions, we continue to employ 
the notations and definitions used in Sections~\ref{intro-section} and
\ref{prelim-section}.

\begin{lemma}\label{nov27-21}
Let $D$ be a weighted oriented graph and let
$I(D)=\bigcap_{i=1}^m\mathfrak{q}_i$ be the irreducible decomposition
of $I(D)$. If there exists
a vertex $v\in V^+(D)$ that is neither a source nor a sink, then
$\bigcap_{i=1}^m\mathfrak{q}_i^2\not\subset I(D)^2$. 
\end{lemma}

\begin{proof} There are $u,x$ in $V(D)$ such that $(u,v)$, $(v,x)$
are in $E(D)$. Let $V(D)=\{t_1,\ldots,t_s\}$ be the vertex 
set of $D$ and let $w_i$ be the weight of $t_i$. We may assume
$u=t_1$, $v=t_2$, $x=t_3$, and $w_2\geq 2$. Then, the monomials
$g_1:=t_1t_2^{w_2}$ and
$g_2:=t_2t_3^{w_3}$ are in $\mathcal{G}(I(D))$, the minimal generating set of $I(D)$.
We set $f=t_1t_2^{w_2}t_3^{w_3}$. We claim that $f\notin I(D)^2$. We
argue by contradiction assuming that $f\in I(D)^2$. There are three
cases to consider.

\begin{enumerate}
\item[(I)] $(t_1,t_3)\in E(D)$, i.e., $g_3:=t_1t_3^{w_3}\in I(D)$. Then, there are
$f_1,f_2\in\{g_1,g_2,g_3\}$ and $t^\delta\in S$ such that
$f=t^\delta f_1f_2$. Clearly $g_i^2$ does not divide $f$ for
$i=1,2,3$. Thus, $f_1\neq f_2$. Then
$$
f=t_1t_2^{w_2}t_3^{w_3}=
\begin{cases}
t^\delta(t_1t_2^{w_2})(t_2t_3^{w_3})& \mbox{or},\\
t^\delta(t_1t_2^{w_2})(t_1t_3^{w_3})& \mbox{or},\\
t^\delta(t_2t_3^{w_3})(t_1t_3^{w_3}).&
\end{cases}
$$
\item[(II)] $(t_3,t_1)\in E(D)$, i.e., $g_3:=t_3t_1^{w_1}\in I(D)$. Then, there are
$f_1,f_2\in\{g_1,g_2,g_3\}$ and $t^\delta\in S$ such that
$f=t^\delta f_1f_2$. Clearly $g_i^2$ does not divide $f$ for
$i=1,2,3$. Thus, $f_1\neq f_2$. Then
$$
 f=t_1t_2^{w_2}t_3^{w_3}=
\begin{cases}
t^\delta(t_1t_2^{w_2})(t_2t_3^{w_3})& \mbox{or},\\
t^\delta(t_1t_2^{w_2})(t_3t_1^{w_1})& \mbox{or},\\
t^\delta(t_2t_3^{w_3})(t_3t_1^{w_1}).
\end{cases}
$$
\item[(III)] $(t_1,t_3)$ and $(t_3,t_1)$ are not in $E(D)$, that is, 
$t_1t_3^{w_3}\notin I(D)$ and $t_3t_1^{w_1}\notin
I(D)$. Then, there are
$f_1,f_2\in\{g_1,g_2\}$ and $t^\delta\in S$ such that
$f=t^\delta f_1f_2$. Clearly $g_i^2$ does not divide $f$ for $i=1,2$.
Thus, $f_1\neq f_2$. Then
$f=t_1t_2^{w_2}t_3^{w_3}=t^\delta(t_1t_2^{w_2})(t_2t_3^{w_3})$. 
\end{enumerate}
\quad In each of the three cases, recalling that $w_i\geq 1$ for all
$i$, we get a contradiction. This proves that $f\notin
I(D)^2$. Let $\mathfrak{q}=\mathfrak{q}_i$ be any irreducible
component of $I(D)$. Next we show that $f\in\mathfrak{q}^2$, and
consequently $f\in\bigcap_{i=1}^m\mathfrak{q}_i^2$. Let
$\mathcal{G}(\mathfrak{q})$ be the minimal generating set of $\mathfrak{q}$. By 
Theorem~\ref{pitones-reyes-toledo}, every monomial in
$\mathcal{G}(\mathfrak{q})$ has the form
$t_p^{\ell_p}$ for some $1\leq p\leq s$, $\ell_p\geq 1$ and 
$t_p^{\ell_p}\in\{t_p,t_p^{w_p}\}$ (cf.
Lemma~\ref{duality-of-exponents}). 
As $t_1t_2^{w_2}$ and
$t_2t_3^{w_3}$ are in $\mathcal{G}(I(D))$, we can write
\begin{align*}
t_1t_2^{w_2}&= t^\delta t_j^{\ell_j}\mbox{ for some }t^\delta\in
S\mbox{ and }t_j^{\ell_j}\in \mathcal{G}(\mathfrak{q}),\mbox{ then
}t_j^{\ell_j}\in\{t_j,t_j^{w_j}\},
\\ t_2t_3^{w_3}&=t^\gamma t_n^{\ell_n}\mbox{ for some }t^\gamma\in
S\mbox{ and }t_n^{\ell_n}\in \mathcal{G}(\mathfrak{q}),\mbox{ then
}t_n^{\ell_n}\in\{t_n,t_n^{w_n}\}.
\end{align*}
\quad There are two cases to consider.
\begin{enumerate}
\item[(A)] Assume that $t_2\notin \mathcal{G}(\mathfrak{q})$. There are two
subcases to consider
\begin{enumerate}
\item[(A.1)] $t_2$ does not divide $t^\gamma$. Then, $t_n^{\ell_n}=t_2$,
a contradiction because $t_2\notin\mathcal{G}(\mathfrak{q})$.

\item[(A.2)] $t_2$
divides $t^\gamma$. Then, $t_3^{w_3}=(t^\gamma/t_2)t_n^{\ell_n}$, and
$t_n^{\ell_n}=t_3$ or $t_n^{\ell_n}=t_3^{w_3}$. 
\begin{enumerate}
\item[(A.2.1)] $t_1$ divides $t^{\delta}$. Then,
$t_2^{w_2}=(t^\delta/t_1)t_j^{\ell_j}$, and 
$t_j^{\ell_j}=t_2^{w_2}$ since $t_2\notin\mathcal{G}(\mathfrak{q})$. Thus,
$f=t_1t_2^{w_2}t_3^{w_3}\in\mathfrak{q}^2$. 

\item[(A.2.2)] $t_1$ does not divide $t^{\delta}$. Then,
$t_j^{\ell_j}=t_1$. Thus, $f=t_1t_2^{w_2}t_3^{w_3}\in\mathfrak{q}^2$. 
\end{enumerate}
\end{enumerate}
\item[(B)] Assume that $t_2\in \mathcal{G}(\mathfrak{q})$. Then, 
$t_2^2\in\mathfrak{q}^2$, and $f=t_1t_2^{w_2}t_3^{w_3}\in\mathfrak{q}^2$ because $w_2\geq
2$.
\end{enumerate}
\quad Therefore, $f\in(\bigcap_{i=1}^m\mathfrak{q}_i^2)\setminus I(D)^2$
and the proof is complete.
\end{proof}

We come to one of our main results.

\begin{theorem}\label{I2=I(2)}  
Let $D$ be a weighted oriented graph and let $G$ be its underlying 
graph. Then, $I(D)^2=I(D)^{(2)}$ if and only if the following
two conditions hold:
\begin{enumerate}
\item[{\rm (i)}] Every vertex in $V^+(D)$ is a sink;
\item[{\rm (ii)}] $G$ has no triangles.
\end{enumerate}
\end{theorem}

\begin{proof} $\Rightarrow$) (i) We argue by contradiction assuming there
is $v$ in $V^+(D)$ which is not a sink. Note that $v$ is not a
source because all sources of $D$ have weight $1$. Let
$\mathfrak{p}_1,\ldots,\mathfrak{p}_r$ be the minimal primes of $I(D)$ and let $\mathfrak{q}_i$ be the 
$\mathfrak{p}_i$-primary component of $I(D)$ for $i=1,\ldots,r$. Then,
by Theorem~\ref{pitones-reyes-toledo}, $\mathfrak{q_i}$ is an irreducible
component of $I(D)$ for $i=1,\ldots,r$ and, by Lemma~\ref{nov1-21}, 
$I(D)^{(2)}=\bigcap_{i=1}^r\mathfrak{q}_i^2$. Thus, 
by Lemma~\ref{nov27-21}, there is
$f\in I(D)^{(2)}\setminus I(D)^2$, a contradiction. 

(ii) By part (i) all vertices of $V^+(D)$ are sinks. Therefore, by
Lemma~\ref{WOG-lemma-mandal}(a), the ideal $I(D)$ has no embedded primes. If $G$ has a
triangle with vertices $v_1,v_2,v_3$, then one has that $v_1v_2v_3\in
I(G)^{(2)}\setminus I(G)^2$
\cite[Proposition~4.10]{symbolic-powers-survey} and $I(G)^2\subsetneq
I(G)^{(2)}$. Then, by
Lemma~\ref{WOG-lemma-mandal}(b),
$I(D)^2\subsetneq I(D)^{(2)}$, a contradiction. Hence, $G$ has no
triangles.

$\Leftarrow$) As $D$ satisfies (ii), by \cite[Theorem~4.13]{symbolic-powers-survey}, one has
$I(G)^2=I(G)^{(2)}$. Hence, using that $D$ satisfies (i) and applying    
Lemma~\ref{WOG-lemma-mandal}(b), we get $I(D)^2=I(D)^{(2)}$.
\end{proof}

We characterize the equality
of ordinary and symbolic powers of $I(D)$. Mandal and Pradhan showed that conditions (a) and (b) of
Theorem~\ref{In=I(n)} are sufficient conditions for the equality
of ordinary and symbolic powers of $I(D)$ \cite[Corollary 3.8]{Mandal-Pradhan}. 

\begin{theorem}\label{In=I(n)}  
Let $D$ be a weighted oriented graph and let $G$ be its underlying 
graph. Then, $I(D)^n=I(D)^{(n)}$ for all $n\geq 1$ if and only if the following
two conditions hold:
\begin{enumerate}
\item[{\rm (a)}] Every vertex in $V^+(D)$ is a sink;
\item[{\rm (b)}] $G$ is a bipartite graph.
\end{enumerate}
\end{theorem}

\begin{proof} $\Rightarrow$) By Theorem~\ref{I2=I(2)} condition (a)
holds. Then, by Lemma~\ref{WOG-lemma-mandal}(b),
$I(G)^n=I(G)^{(n)}$ for all $n\geq 1$. Hence, by
\cite[Theorem~5.9]{ITG}, $G$ is bipartite.

$\Leftarrow$) As $G$ is a bipartite graph, by \cite[Theorem~5.9]{ITG}, $I(G)^n=I(G)^{(n)}$ for all
$n\geq 1$. Hence, using that every vertex of $V^+(D)$ is a sink and
applying Lemma~\ref{WOG-lemma-mandal}(b), we get
that $I(D)^n$ is equal to $I(D)^{(n)}$ for all $n\geq 1$. 
\end{proof}

\begin{corollary}\label{ntf-weighted}  
Let $D$ be a weighted oriented graph and let $G$ be its underlying 
graph. If $I(D)$ has no embedded primes, then the following
conditions are equivalent:
\begin{enumerate}
\item[{\rm (a)}] ${\rm Ass}(I(D)^n)={\rm Ass}(I(D))$ for all $n\geq 1$, i.e., $I(D)$ is
normally torsion-free;
\item[{\rm (b)}] $I(D)^{n}=I(D)^{(n)}$ for all $n\geq 1$;
\item[{\rm (c)}] Every vertex in $V^{+}(D)$ is a sink and $G$ is
bipartite.
\end{enumerate}
\end{corollary}

\begin{proof} By Proposition~\ref{ntf-char}, conditions (a) and (b)
are equivalent and, by Theorem~\ref{In=I(n)}, conditions (b) and (c)
are equivalent. 
\end{proof}

\section{Equality of ordinary and symbolic powers of monomial
ideals}\label{symbolic-section2}

In this section we give necessary conditions for the equality of 
ordinary and symbolic powers of monomial ideals with a minimal
irreducible decomposition. Then, we classify when 
the dual of the edge ideal of a weighted oriented graph is 
normally torsion-free. To avoid repetitions, we continue to employ 
the notations and definitions used in Sections~\ref{intro-section} and
\ref{prelim-section}.

\begin{theorem}\label{dec12-21} Let $I$ be a monomial ideal of $S$
with a minimal irreducible decomposition 
$I=\mathfrak{q}_1\cap\cdots\cap\mathfrak{q}_m$, let $\alpha_i$ 
be the vector in $\mathbb{N}^s\setminus\{0\}$ such that 
$\mathfrak{q}_i=\mathfrak{q}_{\alpha_i}$, and let 
$B$ be the $s\times m$ matrix with column vectors
$\alpha_1^{-1},\ldots,\alpha_m^{-1}$. If $I^{n}=I^{(n)}$ for all
$n\geq 1$, then the following hold:
\begin{enumerate}
\item[(a)] $\overline{I^n}=
\overline{\mathfrak{q}_1^n}\cap\cdots\cap\overline{\mathfrak{q}_m^n}$
for all $n\geq 1$; 
\item[(b)] ${\rm NP}(I)=\mathcal{Q}(B)$, that is, ${\rm NP}(I)={\rm
IP}(I)$; 
\item[(c)] The vertices of $\mathcal{Q}(I)$ are precisely
$\alpha_1^{-1},\ldots,\alpha_m^{-1}$.
\end{enumerate}
\end{theorem}

\begin{proof} (a) The inclusion ``$\subset$'' is clear. To show the
other inclusion take
$t^a\in\bigcap_{i=1}^m\overline{\mathfrak{q}_i^n}$. Hence, by the
description of the integral closure given in Lemma~\ref{icd}, 
for each $1\leq i\leq m$ 
there is $p_i\in\mathbb{N}_+$ such that
$(t^a)^{p_i}\in\mathfrak{q}_i^{np_i}$. Let $p$ be the least common
multiple of $p_1,\ldots,p_m$. Then, for each $i$ we can write
$p=k_ip_i$ for some $k_i\in\mathbb{N}_+$, and consequently
$$ 
(t^a)^p=((t^a)^{p_i})^{k_i}\in(\mathfrak{q}_i^{np_i})^{k_i}=
\mathfrak{q}_i^{np_ik_i}=\mathfrak{q}_i^{np}.
$$
\quad By Lemma~\ref{anoth-one-char-spow-general}, $I$ has no 
embedded primes because in particular we are
assuming $I^n=I^{(n)}$ for 
$n=1$. Then, also by Lemma~\ref{anoth-one-char-spow-general}, 
$(t^a)^p\in I^{(np)}=I^{np}=(I^n)^p$, and we get
$t^a\in\overline{I^n}$.

(b) By part (a) and \cite[Theorem~7.6]{intclos}, we get that ${\rm
NP}(I)={\rm IP}(I)$. 

(c) Since $I$ has no 
embedded primes, by \cite[Theorem~7.1]{intclos} and part (b), 
$\alpha_1^{-1},\ldots,\alpha_m^{-1}$ are vertices of $\mathcal{Q}(I)$
and we have
\begin{align}\label{dec13-21}
&{\rm NP}(I)=\mathcal{Q}(B)=\textstyle
H^+(\alpha_1^{-1},1)\cap\cdots\cap H^+(\alpha_m^{-1},1)\cap
H_{e_1}^+\cap\cdots\cap H_{e_s}^+,
\end{align}
where $H^+(\alpha_i^{-1},1)=\{x\mid \langle x,
\alpha_i^{-1}\rangle\geq 1\}$ and $H_{e_i}^+=\{x\mid \langle x,
e_i\rangle\geq 0\}$ are closed halfspaces, $e_i$ is the $i$-th unit
vector in $\mathbb{R}^s$, and $\langle\
,\, \rangle$ is the standard inner
product on $\mathbb{R}^s$. Let $\mathcal{G}(I)=\{t^{v_1},\ldots,t^{v_q}\}$ be the minimal generating
set of $I$ and let $\beta$ be any vertex of $\mathcal{Q}(I)$. By
\cite[Corollary~1.1.49]{monalg-rev}, $\beta$ is a basic feasible
solution---in the sense of \cite[Definition~1.1.48]{monalg-rev}---for the
system 
$$
x\geq 0,\ \langle x, v_i\rangle\geq 1,\ 
i=1,\ldots,q
$$
of linear constraints that represent $\mathcal{Q}(I)$. Thus, there
are $v_{j_1},\ldots,v_{j_\ell}, e_{k_1},\ldots,e_{k_t}$ 
linearly independent vectors such that the $v_{j_i}$'s are in
$\{v_1,\ldots,v_q\}$, the $e_{k_i}$'s are in
$\{e_1,\ldots,e_s\}$, $s=\ell+t$, and $\beta$ satisfies the following linear constraints
\begin{align}
&\langle x,v_{j_i}\rangle=1,\
\langle x,e_{k_p}\rangle=0\mbox{ for all }i,p,\label{1st-eq}\\
&\langle x,v_i\rangle\geq 1\mbox{ for all }i,\mbox{ and }
x\geq 0.\label{2nd-eq}
\end{align}
\quad Let $H(\beta,1)$ be the hyperplane $\{x\mid \langle x,
\beta\rangle=1\}$. Setting $F_\beta:=H(\beta,1)\cap{\rm NP}(I)$ we claim that 
$F_\beta$ is a facet of ${\rm NP}(I)$ in the sense of
\cite[Definition~1.1.8]{monalg-rev}, that is, $F_\beta$ is a face of
${\rm NP}(I)$ of dimension $s-1$. By Eq.~\eqref{2nd-eq}, one has ${\rm
NP}(I)\subset H^+(\beta,1)$, and $F_\beta\neq\emptyset$ because, by
Eq.~\eqref{1st-eq},   
at least one of the $v_{j_i}$'s belongs to $F_\beta$. Thus, $H(\beta,1)$ is a
supporting hyperplane of ${\rm NP}(I)$. Hence, to prove the claim, it suffices to notice that
the set 
$$
\mathcal{B}=\{v_{j_1}+e_{k_1},\ldots,v_{j_1}+e_{k_t},v_{j_2},\ldots,v_{j_\ell}\}
$$
is linearly independent, $\mathcal{B}\subset F_\beta$, and
$|\mathcal{B}|=s-1$. In particular $H(\alpha_i^{-1},1)\cap{\rm NP}(I)$ is a facet of
${\rm NP}(I)$ for each $i$ since $\alpha_i^{-1}$ is a vertex of
$\mathcal{Q}(I)$ for each $i$. Using Eq.~\eqref{dec13-21} together with
\cite[Theorem~3.2.1]{webster}, we get that either 
$F_\beta=H(\alpha_i^{-1},1)\cap{\rm NP}(I)$ for some $1\leq i\leq m$ or
$F_\beta=H_{e_i}\cap{\rm NP}(I)$ for some $1\leq i\leq s$.

Case (I) Assume that $F_\beta=H(\alpha_i^{-1},1)\cap{\rm NP}(I)$ for
some $1\leq i\leq m$. As $\mathcal{B}\subset F_\beta$, $\alpha_i^{-1}$
satisfies the system of linear equations
\begin{align}
&\langle v_{j_1}+e_{k_1},x\rangle=1,\ldots,\langle
v_{j_1}+e_{k_t},x\rangle=1,\nonumber\\
&\langle v_{j_2},x\rangle=1,\ldots,\langle
v_{j_\ell},x\rangle=1,\nonumber
\end{align}
and also satisfies $\langle
v_{j_1},x\rangle=1$ because $v_{j_1}\in F_\beta$ by
Eq.~\eqref{1st-eq}. It follows that $\alpha_i^{-1}$ satisfies the linear system of
Eq.~\eqref{1st-eq} and since this system has $\beta$ as its unique solution, 
we get $\beta=\alpha_i^{-1}$.

Case (II) Assume that $F_\beta=H_{e_i}\cap{\rm NP}(I)$ for some
$1\leq i\leq s$. As $\mathcal{B}\subset F_\beta$, $e_i$
satisfies the system of linear equations
\begin{align}
&\langle v_{j_1}+e_{k_1},x\rangle=0,\ldots,\langle
v_{j_1}+e_{k_t},x\rangle=0,\nonumber\\
&\langle v_{j_2},x\rangle=0,\ldots,\langle
v_{j_\ell},x\rangle=0,\nonumber
\end{align}
and also satisfies $\langle
v_{j_1},x\rangle=0$ because $v_{j_1}\in F_\beta$ by
Eq.~\eqref{1st-eq}. It follows that $e_i$ must be $0$, a
contradiction. Thus, this case cannot occur.
\end{proof}

\begin{corollary}\label{ntf-J} 
Let $J(D)$ be the dual of the edge ideal $I(D)$  of
a weighted oriented graph $D$. Then, $J(D)^n=J(D)^{(n)}$ for all $n\geq 1$ if 
and only if $J(D)$ is a normal ideal and ${\rm NP}(J(D))={\rm IP}(J(D))$.
\end{corollary}

\begin{proof} $\Rightarrow$) The irreducible components
$I_1,\ldots,I_p$ of $J(D)$ are of the form $(t_i,t_j^{w_j})$ with
$(t_i,t_j)$ an edge of $E(D)$. Hence, the ideal $I_i$ is normal for each $i$
\cite[Proposition 7.5]{intclos}. Then, using Theorem~\ref{dec12-21}, we
get that ${\rm NP}(J(D))={\rm IP}(J(D))$ and 
$$
\overline{J(D)^{n}}=\textstyle\bigcap_{i=1}^p\overline{I_i^n}=
\bigcap_{i=1}^p{I_i^n}=J(D)^{(n)}=J(D)^n
$$
for all $n\geq 1$, that is, $J(D)$ is normal.

$\Leftarrow$) As ${\rm NP}(J(D))={\rm IP}(J(D))$, by
\cite[Theorem~7.6]{intclos} and Lemma~\ref{anoth-one-char-spow-general}, we obtain  
$$\overline{J(D)^{n}}=\textstyle\bigcap_{i=1}^p\overline{I_i^n}=
\bigcap_{i=1}^p{I_i^n}=J(D)^{(n)}
$$
for all $n\geq 1$. Hence, as $J(D)$ is normal, we obtain 
that $J(D)^n=J(D)^{(n)}$ for all $n\geq 1$. 
\end{proof}

\section{Examples}\label{examples-section}
\begin{example}\label{ntf-example}
Let $D$ be the weighted oriented graph whose edge ideal is 
$$
I(D)=(t_1t_2^2,\
t_3t_2^2,\, t_3t_4^2,\, t_1t_4^2)=(t_1,\,
t_3)\textstyle\cap(t_2^2,\, t_4^2).
$$
\quad The underlying graph $G$ of $D$ is bipartite, 
$V^+(D)=\{t_2,\, t_4\}$ and all vertices of $V^+(D)$ are sinks. Then,
$I(D)^{n}=I(D)^{\langle n\rangle}=I(D)^{(n)}$ for all $n\geq 1$, 
see \cite[Corollary 3.8]{Mandal-Pradhan} and Theorem~\ref{In=I(n)}. 
Let $\overline{I(D)}$ be the integral closure of $I(D)$. By
Lemma~\ref{icd}, $g=t_1t_2t_4\in \overline{I(D)}\setminus I(D)$ because 
$$
f^2=(t_1t_2t_4)^2=(t_1t_2^2)(t_1t_4^2)\in I(D)^2.
$$
\quad Thus, $I(G)$ is integrally closed but $I(D)$ is not, i.e., 
being integrally closed is not preserved by making the change of 
variables $t_i\rightarrow t_i^{w_i}$ for all $i$ in the edge ideal
$I(G)$. The vertices of $\mathcal{Q}(I(D))$ are $(1,0,1,0)$ and
$(0,1/2,0,1/2)$ and they correspond to the irreducible components of
$I(D)$ by Theorem~\ref{dec12-21}. The dual of $I(D)$ is
given by 
$$
J(D)=(t_1,t_2^2)\cap(t_3,t_2^2)\cap(t_3,t_4^2)\cap(t_1,t_4^2)=(t_1t_3,\,
t_2^2t_4^2).
$$
\quad Using \textit{Normaliz} \cite{normaliz2} it follows that $J(D)$
is normal and 
${\rm NP}(J(D))={\rm IP}(J(D))$. Hence, by Corollary~\ref{ntf-J},
$J(D)^n=J(D)^{(n)}$ 
for all $n\geq 1$. The vertices of $\mathcal{Q}(J(D))$ are 
$$
(1,1/2,0,0),\, (0,1/2,1,0),\, (0,0,1,1/2),\, (1,0,0,1/2),
$$
and they correspond to the irreducible components of
$J(D)$ by Theorem~\ref{dec12-21}.
\end{example}

\begin{example}\label{ntf-m}
Let $D$ be the weighted oriented graph whose edge ideal is 
$$
I(D)=(t_1t_2^2,\
t_2t_3^2,\, t_3t_1^2)=(t_1^2,\,
t_2)\textstyle\cap(t_1,\, t_3^2)\textstyle\cap(t_2^2,\, t_3)\textstyle\cap(t_1^2,\,
t_2^2,\, t_3^2).
$$
\quad The underlying graph $G$ of $D$ is a triangle and $V^+(D)=\{t_1,\,
t_2,\, t_3\}$. Then,
$I(D)^{\langle n\rangle}=I(D)^n$ for all $n\geq 1$ (see
Lemma~\ref{powers-m}) and $I(D)^{(1)}\neq I(D)$. Note that
$\mathfrak{m}=(t_1,\,t_2,\,t_3)\in{\rm Ass}(I(D))$. This follows from
the above irreducible decomposition of $I(D)$ or directly using Proposition~\ref{dec5-21}.
\end{example}

\begin{example}\label{triangle}
Let $D$ be the weighted oriented graph whose edge ideal is 
$$
I(D)=(t_1t_2^2,\
t_2t_3,\, t_1t_3)=(t_1,\,
t_2)\textstyle\cap(t_1,\, t_3)\textstyle\cap(t_2^2,\, t_3).
$$
\quad The underlying graph $G$ of $D$ is a triangle, $V^+(D)=\{t_2\}$,
and $t_1t_2^2t_3\in I(D)^{(2)}\setminus I(D)^2$.
\end{example}

\begin{example}\label{triangle1}
Let $D$ be the weighted oriented graph whose edge ideal is 
$$
I(D)=(t_3t_1^2,\
t_2t_1^2,\, t_2t_3)=(t_1^2,\,
t_2)\textstyle\cap(t_1^2,\, t_3)\textstyle\cap(t_2,\, t_3).
$$
\quad The underlying graph $G$ of $D$ is a triangle, 
$V^+(D)=\{t_1\}$, and $t_1^2t_2t_3\in I(D)^{(2)}\setminus I(D)^2$.
\end{example}

\begin{example}\label{star}
Let $D$ be the weighted oriented graph whose edge ideal is 
$$
I(D)=(t_1t_2^2,\, t_2t_3)=(t_2)\textstyle\cap(t_2^2,\,
t_3)\textstyle\cap(t_1,\, t_3).
$$
\quad The localizations at the maximal associated primes of $I(D)$ are
$$
S\textstyle\cap(I(D)S_{(t_2,\, t_3)})=(t_2^2,\, t_2t_3)=
(t_2)\textstyle\cap(t_3,\,
t_2^2)\mbox{ and }S\textstyle\cap(I(D)S_{(t_1,\, t_3)})=(t_1,\,
t_3)
$$
and, by Proposition~\ref{nov6-21}, $I(D)^{\langle n\rangle}=(t_2^2,\,
t_2t_3)^n\cap(t_1,\,
t_3)^n$. The underlying graph $G$ of $D$ is a star, 
$V^+(D)=\{t_2\}$, $t_1t_2^2t_3\in I(D)^{(2)}\setminus I(D)^2$, and
one has the equality $I(D)^{\langle n\rangle}=I(D)^n$ for all $n\geq 1$
because $G$ is a star \cite[Theorem~4.12]{Mandal-Pradhan1}. 
\end{example}

\begin{appendix}

\section{Procedures}\label{Appendix}

\begin{procedure}\label{procedure1}
Computing the symbolic powers of a monomial ideal and its irreducible
decomposition using \textit{Macaulay}$2$ \cite{mac2}. This procedure corresponds to
Example~\ref{ntf-m}. One can compute other examples by changing the
polynomial ring $S$ and the generators of the ideal $I$.
\begin{verbatim}
restart
load "SymbolicPowers.m2"
S=QQ[t1,t2,t3,t4,t5]
--Computes I^{(n)} for any monomial ideal I
SPM=(I,k)->intersect(for n from 0 to #minimalPrimes(I)-1 
list localize(I^k,(minimalPrimes(I^k))#n))
I=monomialIdeal(t1*t2^2,t2*t3^2,t3*t1^2)
--Computes the associated primes of an ideal I
ass I
irreducibleDecomposition(I)
n=2
--Computes I^{<n>} using Ass(I)
symbolicPower(I,n)
--Computes I^{(n)} using MinAss(I)
symbolicPower(I,n,UseMinimalPrimes=>true)
--Checks whether or not equality holds
symbolicPower(I,n)==I^n
--Checks whether or not equality holds
symbolicPower(I,n,UseMinimalPrimes=>true)==I^n
mingens(SPM(I,n)/I^n)
localize(I,ideal(t2,t3))
localize(I,ideal(t1,t3))
--checks that powers commute with localization for 
--monomial ideals
(localize(I,ideal(t2,t3)))^n==localize(I^n,ideal(t2,t3))
\end{verbatim}
\end{procedure}

\begin{procedure}\label{procedure2}
Computing the vertices of the covering polyhedron $\mathcal{Q}(I)$ of a monomial ideal
$I$ using \textit{Normaliz} \cite{normaliz2}. This procedure corresponds to
Example~\ref{ntf-m}. One can compute other examples by changing the
linear constraints that define $\mathcal{Q}(I)$.
\begin{verbatim}
amb_space 4
constraints 8
0 1 0 0 >= 0
1 0 0 0 >= 0
0 0 1 0 >= 0
0 0 0 1 >= 0
1 2 0 0 >= 1
0 2 1 0 >= 1
0 0 1 2 >= 1
1 0 0 2 >= 1
SupportHyperplanes
ExtremeRays
VerticesOfPolyhedron
\end{verbatim}
\end{procedure}

\end{appendix}

\section*{Acknowledgments} 
We thank the referee for a careful reading of the paper and for the 
improvements suggested. 
We used \textit{Macaulay}$2$ \cite{mac2} 
to compute symbolic powers of monomial ideals and \textit{Normaliz}
\cite{normaliz2} to compute the vertices of covering polyhedra.

\bibliographystyle{plain}

\end{document}